\documentclass[12pt,a4paper]{amsart}

\usepackage{amssymb}
\usepackage[final]{showkeys}
\usepackage{microtype}
\usepackage{color}
\usepackage{graphicx}
\usepackage[hmargin=3cm,vmargin={3.5cm,4cm}]{geometry}

\newtheorem{te}{Theorem}[section]

\newtheorem{os}[te]{Remark}
\newtheorem{prop}[te]{Proposition}

\numberwithin{equation}{section}

\allowdisplaybreaks

\begin{document}

    \title[Fractional telegraph equations]{Analytical solution of space-time fractional telegraph-type equations involving
    Hilfer and Hadamard derivatives}

	\author{Ram K. Saxena $^1$}
	    \address{${}^1$Department of Mathematics and Statistics, Jai Narain Vyas University,
	    Jodhpur.
	    }
	
    \author{Roberto Garra$^2$}

    \author{Enzo Orsingher$^2$}
    \address{${}^2$Dipartimento di Scienze Statistiche, ``Sapienza'' Universit\`a di Roma.}


    \date{\today}

    \begin{abstract}
	In this paper we consider space-time fractional telegraph equations, where the time derivatives are intended in the sense
	of Hilfer and Hadamard while the space fractional derivatives are meant in the sense of Riesz-Feller.
	We provide the Fourier transforms of the solutions of some Cauchy problems for these fractional equations.
	Probabilistic interpretations of some specific cases are also provided.
	
	\smallskip

    \end{abstract}

    \maketitle

    \section{Introduction}
    In recent years, an increasing interest for the analysis and applications of space and time-fractional generalizations of the telegraph-type equations has been developed in the literature. One of the first works in this direction was the paper by Orsingher and Zhao \cite{tele}
    about the space-fractional telegraph equation and then a number of papers regarding applications in probability \cite{toaldo,ptrf},
    in physics \cite{Baleanu,compte} or merely of mathematical interest \cite{chen,Garg} appeared in the literature. 
    
    In the first part of this paper we 
    study the fractional telegraph equation involving 
    Hadamard-type time-fractional derivatives, that is
    \begin{equation}
    \left(t\frac{\partial}{\partial t}\right)^{\nu}\left(t\frac{\partial}{\partial t}\right)^{\nu}u+2\lambda\left(t\frac{\partial}{\partial t}\right)^{\nu}u= c^2\frac{\partial^2 u}{\partial x^2},
    \end{equation}
    with $x\in \mathbb{R}$ and $t>t_0>0$. We denote with the symbol $\left(t\frac{\partial}{\partial t}\right)^{\nu}u$ the 
    Caputo-type modification of the Hadamard derivative of order $\nu$ recently introduced in \cite{bal} as follows
    \begin{equation}
    \left(t\frac{d}{dt}\right)^{\nu}u(t)= \frac{1}{\Gamma(n-\nu)}\int_{t_0}^t\left(\ln\frac{t}{\tau}\right)^{n-\nu-1}\left(\tau \frac{d}{d\tau}\right)^n
    u(\tau)\frac{d\tau}{\tau},
    \end{equation}
    for $n-1<\nu<n$ and $n\in \mathbb{N}$.
    We will explain the reason why we choose to use this symbol for the Caputo-like Hadamard fractional operator and recall some of its main
    properties in the next section devoted to mathematical preliminaries.
    Within the framework of the analysis of the fractional telegraph equation involving Hadamard derivatives we will present analytical results
    concerning the Fourier transform of the fundamental solution and give some insights regarding the probabilistic meaning of the obtained
    results. 
    
    In the second part of the paper, we consider the generalized telegraph equation involving Hilfer fractional derivatives $ D_{t}^{\gamma,\delta}$ in time and Riesz-Feller fractional derivatives $_{x}D^{\alpha}_{\theta}$ in space. The generalized telegraph equation is the following partial differential equation
    \begin{equation}\label{linea}
    \frac{\partial^2 u}{\partial t^2}+2\lambda\frac{\partial u}{\partial t}= c^2\frac{\partial^2 u}{\partial x^2}-\omega u, \quad \omega >0 
    \end{equation}
    and some fractional generalizations of this equation have been recently studied in a number of papers (see for example \cite{Garg,haub} and 
    the references therein). It is well-known that equation \eqref{linea} arises, for example, in the analysis of transmission lines with losses.

    Here we give a new useful representation formula for the solution of the Cauchy problem for the equation 
    \begin{equation}\label{a}
    D_{t}^{2\gamma,\delta}u(x,t)+2\lambda D_{t}^{\gamma,\delta} u(x,t)= c^2 \ _{x}D^{\alpha}_{\theta} u(x,t)-\omega u(x,t)+F(x,t), \quad x \in \mathbb{R}, t\geq 0
    \end{equation}
    We observe that general results about non-homogeneous generalized fractional telegraph equations involving Hilfer and Riesz-Feller 
    operators have been discussed in \cite{haub}. Here we focus our attention on a specific case, that of equation \eqref{a},
    where the solution can be expressed in terms of linear combination of two-parameter Mittag-Leffler functions. This permits 
    us to give an interesting probabilistic interpretation of the solutions when the parameters $\gamma$ and $\delta$ are suitably
    chosen. 
    We also show that the obtained results extend those appeared so far for the classical and the fractional telegraph equation.

    \section{Preliminaries}
   	For the reader's convenience we first present the necessary mathematical preliminaries about fractional operators and 
   	their main properties used in the next sections. We moreover give a brief recall about the classical telegraph process \cite{Gold,kac} and its
   	relationship with the telegraph equation. 
   	
   	\subsection{Fractional differential operators}
   	The classical Hadamard fractional derivative reads (see for example \cite{e-p})
   	\begin{equation}
   	       ^{R}\left(t\frac{d}{dt}\right)^{\nu}u(t)= \frac{1}{\Gamma(n-\nu)}\left(t \frac{d}{dt}\right)^n\int_{t_0}^t\left(\ln\frac{t}{\tau}\right)^{n-\nu-1}
   	       u(\tau)\frac{d\tau}{\tau},
   	       \end{equation}
   	 where $t>t_0>0$ and $n-1<\nu<n$, with $n\in \mathbb{N}$.\\
   In the recent paper \cite{bal} the Caputo-type regularized counterpart of the Hadamard fractional derivative 
   has been introduced as 
   \begin{equation}
       \left(t\frac{d}{dt}\right)^{\nu}u(t)= \frac{1}{\Gamma(n-\nu)}\int_{t_0}^t\left(\ln\frac{t}{\tau}\right)^{n-\nu-1}\left(\tau \frac{d}{d\tau}\right)^n
       u(\tau)\frac{d\tau}{\tau},
       \end{equation}
        where $t>t_0>0$ and $n-1<\nu<n$, with $n\in \mathbb{N}$.
     In this paper we adopt the symbol $\left(t\frac{d}{dt}\right)^{\nu}$ for the Caputo-type Hadamard 
     fractional derivative in order to highlight the relationship with the
     differential operator $\delta = t\frac{d}{dt}$. Indeed, we observe that, for $\nu = 1 $ this operator coincides by definition with $\delta$.   
   One of the useful properties that we will be employed in the next section is the following one
   \begin{equation}\label{pbe}
   \left(t\frac{d}{dt}\right)^{\nu}\left(\ln\frac{t}{t_0}\right)^{\beta}= \frac{\Gamma(\beta+1)}{\Gamma(\beta+1-\nu)}
   \left(\ln\frac{t}{t_0}\right)^{\beta-\nu}, \quad \mbox{when $t_0>0, \ \beta > -1 / \{0\}$}
   \end{equation}
   which can be proved by simple calculations. 
   We remark that, in the case $\beta = 0$, as can be simply understood in view of the definition (2.2), we have that 
   \begin{equation}
    \left(t\frac{d}{dt}\right)^{\nu} const. = 0.
   \end{equation}
   We refer to \cite{bal} for a complete discussion about the functional setting and main
   properties of this integro-differential operator.
    
    In a series of works (see \cite{Hilfer} and the references therein), Hilfer studied applications of a generalized fractional operator
    having the Riemann--Liouville and the Caputo derivatives as specific cases.
     The Hilfer's derivative is defined as
        \begin{align}
        \nonumber D_{a^+}^{\mu,\nu}u(x,t)&= \frac{1}{\Gamma(\nu(n-\mu))}\int_a^t dy \ (t-y)^{\nu(n-\mu)-1}\frac{d^n}{dy^n}\int_a^y
        \frac{(y-z)^{(1-\nu)(n-\mu)-1}}{\Gamma((1-\nu)(n-\mu)}u(x,z)dz\\
        & =I_{a^+}^{\nu(n-\mu)}D_{a^+}^{\mu+\nu n-\mu \nu}u(x,t), \quad \nu \in [0,1), \ \mu \in (0,1).
        \end{align}
        Clearly $n-1<(1-\nu)(n-\mu)< n$ and $D_{a^+}^{\mu+\nu n-\mu \nu}$ is the Riemann-Liouville fractional derivative and 
        $I_{a^+}^{\nu(n-\mu)}$ is the Riemann-Liouville integral (see e.g. \cite{pod,e-p}). Hereafter we will take for simplicity $a= 0$. 
        
        The Laplace transform of the Hilfer fractional derivative w.r. to the time reads (see \cite{tomo})
        \begin{align}\label{la}
        \mathcal{L}\left(D_{a^+}^{\mu,\nu}u\right)(x,s)&= \int_0^\infty e^{-st} D_{a^+}^{\mu,\nu}u(x,t) dt\\
        \nonumber &= s^{\mu}\tilde{u}(x,s)-\sum_{k=0}^{n-1}s^{n-k-\nu(n-\mu)-1}\frac{d^k}{dt^k}I_{0^+}^{(1-\nu)(n-\mu)} u(x,t)\bigg|_{t=0},
        \end{align}
        where $\left(\mathcal{L}u\right)(s)= \tilde{u}(x,s)$.
       
       We now recall the definition of the Riesz-Feller fractional derivative and for more details we refer, for example, to the encyclopedical 
       book by Samko et al. \cite{e-p}.
       For $0<\alpha\leq 2$ and $|\theta|\leq \min (\alpha, 2-\alpha)$, the Riesz-Feller derivative is defined as
       \begin{align}
       \label{b} &\left(_{x}D^{\alpha}_{\theta}f\right)(x)=\\ & \nonumber =\frac{\Gamma(1-\alpha)}{\pi}\bigg\{
       \sin\frac{\pi(\alpha+\theta)}{2}\int_0^{\infty}\frac{f(x+\xi)-f(x)}{\xi^{1+\alpha}}d\xi+
        \sin\frac{\pi(\alpha-\theta)}{2}\int_0^{\infty}\frac{f(x-\xi)-f(x)}{\xi^{1+\alpha}}d\xi\bigg\}.
       \end{align}
       For $\theta = 0$, the Riesz-Feller derivative becomes the Riesz fractional derivative, that is 
       \begin{equation}
        \left(_{x}D^{\alpha}_{0}f\right)(x)=-\left(-\frac{d^2}{dx^2}\right)^{\alpha/2}
       \end{equation}
       The Fourier transform of \eqref{b} reads
       \begin{equation}\label{fu}
       \mathcal{F}\left(_{x}D^{\alpha}_{\theta}f\right)(\beta)
       = -|\beta|^{\alpha}e^{i  \frac{\theta \pi}{2}sign \beta}
       f^*(\beta)= -\psi^{\theta}_{\alpha}(\beta)f^*(\beta), 
       \end{equation}
    	where $f^*(\beta)= \left(\mathcal{F}f\right)(\beta)$.
    	
    	We note that the fundamental solution to the space-fractional 
    	differential equation involving the Riesz-Feller derivative
    	\begin{align}
    	\nonumber &\frac{\partial p}{\partial t}= _{x}D^{\alpha}_{\theta}p,
    	\end{align}
    	has Fourier transform
    	\begin{equation}
    	p^*(\beta,t)= exp\{-|\beta|^{\alpha} t e^{i  \frac{\theta \pi}{2}}sign \beta \},
    	\end{equation}
    	which is the characteristic function of a stable process 
    	$S_{\alpha}(\sigma, \gamma, \mu; t)$, with $\mu = 0$,
    	$\sigma= \cos \frac{\theta \pi}{2}$ and $\gamma = -\frac{\tan\frac{\pi \theta}{2}}{\tan\frac{\pi \alpha}{2}}$, (see e.g. \cite{zala}).
    	
    	\subsection{The telegraph process}
    	    
    	     The so-called telegraph process, $\mathcal{T}(t)$, $t>0$, describes the random motion of a particle
    	        moving on the real line with finite velocity $c$ and alternating two possible directions
    	        of motions (forward or backward) at Poisson paced times with constant rate $\lambda>0$.
    	        This simple finite-velocity random motion was firstly suggested by the description of particles transport.
    	        In this framework, the first derivation given in literature, as far as we kwow, was given 
    	        by F\"urth in a discussion of several models of fluctuation
    	        phenomena in physics and biology. A similar model was also considered by G.I.Taylor \cite{tay} to treat
    	        turbulent diffusion and then studied in detail
    	        by Sidney Goldstein in \cite{Gold} and Mark Kac in \cite{kac}. This process is
    	        called in the literature \textit{telegraph process} because the probability law of
    	        $\mathcal{T}(t)$ coincides with the fundamental solution of the hyperbolic telegraph equation.
    	        Moreover, as a limiting case it coincides with the transition density
    	        of the classical Brownian motion. 
    	       
    	        In more detail, the classical symmetric telegraph process is defined as
    	            \begin{equation}\label{oooo}
    	                \mathcal{T}(t)=V(0)\int_0^t(-1)^{\mathcal{N}(s)}ds, \qquad t \ge 0,
    	            \end{equation}
    	            where $V(0)$ is a two-valued random variable independent of
    	            the Poisson process $\mathcal{N}(t)$, $t\geq 0$.

    	    	    The component of the
    	            unconditional distribution of the telegraph process concentrated inside the interval $(-ct,+ct)$, is given by
    	            \begin{align}
    	                P\{\mathcal{T}(t)\in dx\}
    	                 &= dx \frac{e^{-\lambda t}}{2c}\left[\lambda I_0\left(\frac{\lambda}{c}
    	                \sqrt{c^2t^2-x^2}\right)+\frac{\partial}{\partial t}I_0\left(\frac{\lambda}{c}
    	                \sqrt{c^2t^2-x^2}\right)
    	                \right], \qquad  |x|<ct.
    	                \label{tp}
    	            \end{align}
    	            The component of the unconditional distribution that pertains to the Poisson probability of no
    	            changes of directions is concentrated on the boundary i.e. $x= \pm ct$,
    	            \begin{equation}
    	                P\{\mathcal{T}(t)=\pm ct\}=\frac{e^{-\lambda t}}{2}.
    	                \label{diso1}
    	            \end{equation}
    	            Hence we are able to give in explicit form the density $f(x,t)$ of the distribution of
    	            $\mathcal{T}(t)$, that is
    	            \begin{align}
    	            \nonumber f(x,t) &= \frac{e^{-\lambda t}}{2c}\left[\lambda I_0\left(\frac{\lambda}{c}
    	                \sqrt{c^2t^2-x^2}\right)+\frac{\partial}{\partial t}I_0\left(\frac{\lambda}{c}
    	                \sqrt{c^2t^2-x^2}\right)\right]1_{(-ct,+ct)}(x)\\
    	            \nonumber &+\frac{e^{-\lambda t}}{2}[\delta(ct-x)+\delta(ct+x)],
    	            \end{align}
    	            where $\Theta(\cdot)$ is the Heaviside function and $\delta(\cdot)$ is the Dirac delta function.
    	            We remark  that the component of the distribution of the
    	            telegraph process concentrated in $(-ct,+ct)$, given by formula \eqref{tp}, is the solution to the Cauchy
    	            problem
    	            \begin{equation}
    	                \label{te1}
    	                \begin{cases}
    	                    \frac{\partial^2 p}{\partial t^2}+2\lambda \frac{\partial p}{\partial
    	                    t}= c^2 \frac{\partial^2 p}{\partial x^2},\\
    	                    p(x,0)=\delta(x),\\
    	                    \frac{\partial p}{\partial t}(x,t)\bigg|_{t=0}=0.
    	                \end{cases}
    	            \end{equation}

    \section{The fractional telegraph equation involving Hadamard fractional derivatives}
    
    In this section we study the following time-fractional telegraph equation involving Caputo-like regularized 
    Hadamard derivatives  introduced in the previous section
    \begin{equation}\label{opom}
           \left(t\frac{\partial}{\partial t}\right)^{\nu}\left(t\frac{\partial}{\partial t}\right)^{\nu}u+2\lambda\left(t\frac{\partial}{\partial t}\right)^{\nu}u= c^2\frac{\partial^2 u}{\partial x^2}, \quad \nu \in (0,1)
    \end{equation}
    It corresponds to a generalization of the 
    telegraph equation with variable coefficients. The main results presented in this section are strongly 
    inspired by the previous one obtained by Beghin and Orsingher in \cite{ptrf}. Indeed we 
    will do an Ansatz about the form of the Fourier transform of the fundamental solution that is based on the 
    results proved in \cite{ptrf} and then we verify our assumption. We will also explain the heuristic reasonableness
    of our statement.
    \begin{te}
    The Fourier transform of the solution to the Cauchy problem
    \begin{equation}
    \begin{cases}
       \left(t\frac{\partial}{\partial t}\right)^{\nu}\left(t\frac{\partial}{\partial t}\right)^{\nu}u+2\lambda\left(t\frac{\partial}{\partial t}\right)^{\nu}u= c^2\frac{\partial^2 u}{\partial x^2}, \quad \nu \in (0,1), \ t_0>0 \\
        u(x,t_0)= \delta(x),\\
        \frac{\partial u}{\partial t}\big|_{t=t_0}=0,
        \end{cases}
    \end{equation}
    is given by
    \begin{align}\label{rombo}
    u^*(\beta,t)= \frac{1}{2}\bigg[\left(1+\frac{\lambda}{\sqrt{\lambda^2-c^2\beta^2}}\right)E_{\nu,1}\left(\eta_1 ln^{\nu}\left(\frac{t}{t_0}\right)\right)
    \\
    \nonumber +\left(1-\frac{\lambda}{\sqrt{\lambda^2-c^2\beta^2}}\right)E_{\nu,1}\left(\eta_2 ln^{\nu}\left(\frac{t}{t_0}\right)\right)\bigg],
    \end{align}
    where 
    \begin{equation}
    \eta_1 = -\lambda +\sqrt{\lambda^2-c^2\beta^2}, \quad 
    \eta_2= -\lambda -\sqrt{\lambda^2-c^2\beta^2} 
    \end{equation}
    \end{te}
    \begin{proof}
    Let us introduce the operator
    $$\mathcal{H}_{\nu}= \left(t\frac{\partial}{\partial t}\right)^{\nu}\left(t\frac{\partial}{\partial t}\right)^{\nu}+2\lambda \left(t\frac{\partial}{\partial t}\right)^{\nu}.$$
    By using the property \eqref{pbe} of the Caputo-like Hadamard derivative, we have that
    \begin{align}
    \nonumber &\left(t\frac{\partial}{\partial t}\right)^{\nu} E_{\nu,1}\left(\eta_1 ln^{\nu}\left(\frac{t}{t_0}\right)\right) 
    = 
    \sum_{k= 1}^{\infty}\frac{\Gamma(\nu k+1)}{\Gamma(\nu k+1-\nu)}
    \frac{\left(ln(t/t_0)\right)^{\nu k-\nu}\eta_1^k}{\Gamma(\nu k+1)}\\
     \nonumber &= \eta_1 \sum_{k'= 0}^{\infty}
       \frac{\left(ln(t/t_0)\right)^{\nu k'}\eta_1^{k'}}{\Gamma(\nu k'+1)}=   \eta_1
                          E_{\nu,1}\left(\eta_1 ln^{\nu}\left(\frac{t}{t_0}\right)\right),
    \end{align}
    where we used the fact that the Caputo-type Hadamard derivative of a constant is zero. On the basis of this result we have by straight calculation that 
    \begin{equation}
    \mathcal{H}_{\nu}u^*(\beta,t)= -c^2\beta^2 u^*(\beta, t),
    \end{equation}
    as claimed.
    \end{proof}
    Result (3.3) is heuristically motivated by the fact that the Fourier transform of the solution of the fractional telegraph equation involving 
    Hadamard derivatives can be obtained from the solution of the fractional problem involving Caputo derivatives by means 
    of the deterministic time-change $t\rightarrow \ln t$. 
    
   This observation implies that, for $\nu=1/2$, formula \eqref{rombo} becomes the charachteristic function of a telegraph
   process at a time-changed Brownian motion, as suggested by the results of section 4 of \cite{ptrf}.
    In fact, for $\nu = 1/2$, formula \eqref{rombo} can be written as
    \begin{align}
        u^*(\beta,t)=& \frac{\lambda}{2\sqrt{\pi ln\frac{t}{t_0}}}\int_0^{\infty}e^{-\frac{z^2}{4 ln\frac{t}{t_0}}-\lambda z}
            \bigg\{\frac{e^{z\sqrt{\lambda^2-c^2\beta^2}}-e^{-z\sqrt{\lambda^2-c^2\beta^2}}}{\sqrt{\lambda^2-c^2\beta^2}}\bigg\}dz\\
            \nonumber &+\frac{1}{2\sqrt{\pi ln\frac{t}{t_0}}}\int_0^{\infty}e^{-\frac{z^2}{4 ln\frac{t}{t_0}}-\lambda z}
                        \bigg\{e^{z\sqrt{\lambda^2-c^2\beta^2}}+e^{-z\sqrt{\lambda^2-c^2\beta^2}}\bigg\}dz
    \end{align}
    that coincides with the characteristic
    function of the process
    \begin{equation}
    W(t)= \mathcal{T}\left(\bigg|B\left(ln\left(\frac{t}{t_0}\right)\right)\bigg|\right),
    \end{equation}
    that is a telegraph proces with random Brownian time with a logarithmic deterministic time-change (see \cite{ptrf}).
    
    We moreover observe that for $\nu \rightarrow 1$, we have studied a fractional generalization of the telegraph equation with variable coefficients
    \begin{equation}\label{laa}
    \left(t\frac{\partial}{\partial t}\right)^2 u+2\lambda \left(t\frac{\partial}{\partial t}\right) u = c^2 \frac{\partial^2 u}{\partial x^2}.
    \end{equation}
    Then, it is simple to prove that the fundamental solution of this telegraph-type equation coincides with the law of the classical telegraph process with the aforementioned deterministic time change.
    
    \smallskip
    
    On the basis of the previous considerations, we now consider the space-time fractional telegraph equation
    \begin{equation}\label{bruno}
     \left(t\frac{\partial}{\partial t}\right)^{\nu}\left(t\frac{\partial}{\partial t}\right)^{\nu}u+2\lambda\left(t\frac{\partial}{\partial t}\right)^{\nu}u= -c^2\left(-\frac{\partial^2}{\partial x^2}\right)^{\alpha/2}u, \quad x\in \mathbb{R}, t>0
    \end{equation}
	involving time-fractional Hadamard derivatives and space-fractional Riesz derivative (see the previous section). Equation \eqref{bruno}
	has been previously studied in \cite{toaldo} in the case where the Hadamard derivatives are replaced with Caputo fractional derivatives.
	In this paper the authors have shown that the fundamental solution of the space-time fractional telegraph-type equation
	\begin{equation}\label{bruno}
	     \frac{\partial^{2\nu}u}{\partial t^{2\nu}}+2\lambda\frac{\partial^{\nu}u}{\partial t^{\nu}}= -c^2\left(-\frac{\partial^2}{\partial x^2}\right)^{\alpha/2}u, \quad \nu\in (0,\frac{1}{2}], \ \alpha \in (0,2]
	\end{equation}
	coincides with the distribution of the composition of a a stable process $S^{\alpha}(t)$, $t>0$ with the positively-valued process
	\begin{equation}
	\mathcal{L}^{\nu}(t)= \inf\{s\geq 0 : \mathcal{H}^{\nu}(s)=  H_1^{2\nu}(s)+(2\lambda)^{1/\nu} H_2^{\nu}(s)\geq t\},\quad t>0,
	\end{equation}
	where $H_1^{2\nu}$ and $H_2^{\nu}$ are independent positively skewed stable processes of order $2\nu$ and $\nu$ respectively.
    With the next proposition we generalize Theorem 4.1 of \cite{toaldo} to the case of equation \eqref{bruno}
    \begin{prop}
    For $\nu \in (0,1/2]$ and $\alpha \in (0,2]$, the fundamental solution of equation \eqref{bruno}
    coincides with the probability law of the process
    \begin{equation}
    W(t)= S^{\alpha}\left(c^2\mathcal{L}^{\nu}\left(ln \ \left(\frac{t}{t_0}\right)\right)\right), \quad t>0
    \end{equation}
    and has Fourier transform 
    \begin{align}
    \nonumber u^{*}(\beta,t)&= \frac{1}{2}\bigg[\left(1+\frac{\lambda}{\sqrt{\lambda^2-c^2|\beta|^{\alpha}}}\right)E_{\nu,1}\left(r_1 ln^{\nu}\ \frac{t}{t_0}\right)\\
    &+\left(1-\frac{\lambda}{\sqrt{\lambda^2-c^2|\beta|^{\alpha}}}\right)E_{\nu,1}\left(r_2 ln^{\nu}\ \frac{t}{t_0}\right),
    \end{align}
    
    where
    \begin{equation}
    r_1= -\lambda +\sqrt{\lambda^2-c^2|\beta|^{\alpha}},
    \quad  r_1= -\lambda -\sqrt{\lambda^2-c^2|\beta|^{\alpha}}
    \end{equation}
    
    \end{prop} 
	We neglect the complete proof of this proposition because it consists of a simple combination of the arguments of the proof of Theorem 4.1 in \cite{toaldo} and of the previous considerations about the role of Hadamard time-fractional derivatives that induces a deterministic logarithmic time-change. However we remark that this result can be easilty generalized to the multidimensional case (where the fractional Laplacian appears) and gives a more general probabilistic interpretation of the fundamental solution of the space-time fractional telegraph equation.

    \subsection{Decomposing the fractional telegraph equation involving Hadamard derivatives}
     Here we show that the fractional equation \eqref{opom} can be easily found by decoupling it into a system of two fractional equations and 
     we provide a brief discussion about the utility of this equivalent formulation.
     By simple calculation it can be proved that \eqref{opom} emerges from the system
     \begin{equation}\label{deco}
     \begin{cases}
    & \left(t\frac{\partial}{\partial t}\right)^{\nu}u= -c\frac{\partial}{\partial x} w\\
      &\left(t\frac{\partial}{\partial t}\right)^{\nu}w+2\lambda w = -c\frac{\partial u}{\partial x}
     \end{cases}
     \end{equation} 
     The utility of the decomposition \eqref{deco} of equation \eqref{laa} is twofold. It gives useful insights on the meaning of the equation 
     \eqref{opom} both from the probabilistic and physical points of view. From the physical points of view this is a generalization of the fractional Cattaneo heat model widely studied in the literature (see for example \cite{povst} and the references therein). The generalization stems from in the fact that, replacing the Caputo derivative with the Hadamard derivative, we obtain a deterministic time-change leading to a slower, logarithmic in time, heat wave propagation.
	On the other hand, we recall that the probability law of the telegraph process can be obtained by solving the coupled system
	  \begin{align}
	      & \frac{\partial}{\partial t}u= -c\frac{\partial w}{\partial x}\\
	       &\frac{\partial}{\partial t}w+2\lambda w = -c\frac{\partial u}{\partial x},
	       \end{align}  
	 where 
	 \begin{equation}
	 u(x,t)=f(x,t)+b(x,t), \quad w(x,t)= f(x,t)-b(x,t),
	 \end{equation}
	 and $f(x,t)$ represents the probability that the particle is near $x$ at time $t$ with forward velocity and $b(x,t)$  the probability that the particle is near $x$ at time $t$ with backward velocity.
    
    \section{The generalized fractional telegraph equation involving Hilfer derivatives}
    
    In this section we study the generalized non-homogeneous fractional telegraph equation involving 
    Hilfer derivatives in time and Riesz-Feller derivatives in space
    \begin{equation}\label{saz}
    D_{t}^{2\gamma,\delta}u(x,t)+2\lambda D_{t}^{\gamma,\delta} u(x,t)= c^2 \ _{x}D^{\alpha}_{\theta} u(x,t)-\omega u(x,t)+F(x,t), \quad x \in \mathbb{R},
     t\geq 0.
    \end{equation}
    In this framework our first general result is given by the following 
    \begin{te}
    Consider the generalized non-homogeneous space-time fractional telegraph equation \eqref{saz} subject to the following constraints
    \begin{equation}
    0<\gamma\leq 1, \quad 0\leq \delta\leq 1, \quad
    0 < \alpha \leq 2 \quad |\theta|\leq \min (\alpha, 2-\alpha).
    \end{equation}
    If the equation \eqref{saz} is subject to the following initial and 
    boundary coinditions
    \begin{align}
    I_{0^+}^{(1-\delta)(2-2\gamma)}u(x,0)= f_1(x),\\
    \frac{d}{dt}I_{0^+}^{(1-\delta)(2-2\gamma)}u(x,t)\bigg|_{t=0}=0,\\
    I_{0^+}^{(1-\delta)(1-\gamma)}u(x,0)=f_2(x),\\
    \lim_{|x|\rightarrow \infty}u(x,t)=0,
        \end{align}
    then its solution reads
    \begin{align}
    & u(x,t)
    =\frac{1}{2\pi(\xi-\eta)}\int_{-\infty}^{+\infty}
	\left[t^{\gamma+2\delta-2\gamma \delta -2}\{E_{\gamma, \gamma+2\delta-2\gamma\delta -1}(\xi t^{\gamma})-
	E_{\gamma, \gamma+2\delta-2\gamma\delta -1}(\eta t^{\gamma})\}\right] f_1^*(\beta)e^{-i\beta x}d\beta \label{remo}\\
	\nonumber &+  \frac{1}{2\pi(\xi-\eta)}\int_{-\infty}^{+\infty}
		\left[t^{\gamma-1}\{E_{\gamma, \gamma+\delta-\gamma\delta }(\xi t^{\gamma})-
		E_{\gamma, \gamma+\delta-\gamma\delta}(\eta t^{\gamma})\}\right] f_2^*(\beta)exp(-i\beta x)d\beta\\
	\nonumber &+\frac{1}{2\pi(\xi-\eta)}\int_0^t\tau^{\gamma-1}
	\int_{-\infty}^{+\infty}F^*(\beta, t-\tau)\left[E_{\gamma, \gamma}(\xi \tau^{\gamma})-E_{\gamma, \gamma}(\eta \tau^{\gamma})\right]exp(-i\beta x) d\beta d\tau,
    \end{align}
    where 
    \begin{equation}
    \xi= -\lambda +\sqrt{\lambda^2-b}, \quad \eta = -\lambda
    -\sqrt{\lambda^2-b},
    \end{equation}
    and 
    $b=\omega +c^2\psi^{\theta}_{\alpha}$.
    \end{te}
    
    \begin{proof}
    If we apply the Laplace transform with respect to the time variable $t$ and Fourier transform with respect to the space variable $x$, by using the initial conditions and formulas \eqref{la} and \eqref{fu}, we obtain
    \begin{align}\label{pr1}
    &s^{2\gamma}\tilde{u}^*(\beta,s)-s^{1-\delta(2-2\gamma)}f_1^*(\beta)
    +2\lambda\left[s^{\gamma}\tilde{u}^*(\beta,s)-s^{-\delta(1-\gamma)}f_2^*(\beta)\right]\\
    \nonumber& = -c^2 \psi_{\alpha}^{\theta}(\beta)\tilde{u}^*(\beta,s)
    -\omega \tilde{u}^*(\beta,s)+\tilde{F}^*(\beta,s),
    \end{align}
    where the symbol $*$ represents the Fourier transform with respect to the space variable $x$ and $\tilde{u}(\beta)$ stands for the Laplace transform with respect to the time variable.
    
    From \eqref{pr1}, we have that
    \begin{equation}\label{prr}
    \tilde{u}^*(\beta,s)= \frac{s^{1-\delta(2-2\gamma)}f_1^*(\beta)
        +2\lambda s^{-\delta(1-\gamma)}f_2^*(\beta)+\tilde{F}^*(\beta,s)
    }{s^{2\gamma}+2\lambda s^{\gamma}+b},
    \end{equation}
    with $b= \omega+ c^2 \psi^{\theta}_{\alpha}(\beta)$.
    
    To invert the Laplace transform, we set
    \begin{align}
      \label{pr2} \Delta(\beta,s)= \frac{s^{1-\delta(2-2\gamma)}}{s^{2\gamma}+2\lambda s^{\gamma}+b}\\
      \Xi(\beta,s)= \frac{s^{-\delta(1-\gamma)}}{s^{2\gamma}+2\lambda s^{\gamma}+b}\\
      \Omega(\beta, s)= \frac{1}{s^{2\gamma}+2\lambda s^{\gamma}+b}.
       \end{align}
    Inverting the Laplace transform in \eqref{pr2},
    we obtain
    \begin{align}
    \nonumber \Delta(\beta,t) &= \mathcal{L}^{-1}\left[\frac{s^{1-\delta(2-2\gamma)}}{s^{2\gamma}+2\lambda s^{\gamma}+b},t\right]\\
    \nonumber &= \frac{1}{\xi-\eta} \mathcal{L}^{-1}\left(
    \left[\frac{s^{1-\delta(2-2\gamma)}}{s^{\gamma}-\xi}-
    \frac{s^{-\delta(1-2\gamma)}}{s^{\gamma}-\eta}\right]\right)\\
    &= \frac{1}{\xi-\eta}\left[t^{\gamma+2\delta-2\gamma\delta -2}\left(E_{\gamma, \gamma+2\delta-2\gamma\delta-1}(\xi t^{\gamma})-E_{\gamma, \gamma+2\delta-2\gamma\delta-1}(\eta t^{\gamma})\right)\right],\label{pr3}
    \end{align}
    where $\eta$ and $\xi$ are the roots of the quadratic equation
    $$t^{2\gamma}+2\lambda t^{\gamma}+b=0.$$
    The functions appearing in \eqref{pr3} are the classical one parameter Mittag-Leffler functions (see for example \cite{gorenflo}
    for a complete review). 
    
    The other terms appearing in \eqref{prr} can be handled in a similar way. In particular, we have that 
    \begin{align}
    \Xi(\beta,t)= \frac{1}{\xi-\eta}\left[t^{\gamma+\delta-\gamma\delta -1}\left(E_{\gamma, \gamma+\delta-\gamma\delta}(\xi t^{\gamma})-E_{\gamma, \gamma+\delta-\gamma\delta}(\eta t^{\gamma})\right)\right]
    \end{align}
    and
    \begin{align}
    \Omega(\beta,t)= \frac{t^{\gamma-1}}{\xi-\eta}[E_{\gamma,\gamma}(\xi t^{\gamma})-E_{\gamma,\gamma}(\eta t^{\gamma})].
    \end{align}
    Taking all these terms together and applying the convolution theorem, we have the Fourier transform 
    of the solution of the Cauchy problem
    \begin{align}
    \nonumber u^*(\beta,t)&= \frac{1}{\xi-\eta}\bigg\{\left[t^{\gamma+2\delta-2\gamma\delta -2}\left(E_{\gamma, \gamma+2\delta-2\gamma\delta-1}(\xi t^{\gamma})-E_{\gamma, \gamma+2\delta-2\gamma\delta-1}(\eta t^{\gamma})\right)\right]f_1^*(\beta)\\
    \nonumber &+\left[t^{\gamma-1}\left(E_{\gamma, \gamma+\delta-\gamma\delta}(\xi t^{\gamma})-E_{\gamma, \gamma+\delta-\gamma\delta}(\eta t^{\gamma})\right)\right]f_2^*(\beta)\\
    \nonumber &+ \int_0^t \tau^{\gamma-1} F^{*}(\beta, t-\tau)
    [E_{\gamma,\gamma}(\xi \tau^{\gamma})-E_{\gamma,\gamma}(\eta \tau^{\gamma})]d\tau\bigg\}.
    \end{align}
    
    Finally by using the inverse Fourier transform we obtain the claimed result.
    \end{proof}
    \begin{os}
    Under the same conditions of Theorem 4.1, if $\theta = 0$, the space-fractional derivative becomes a Riesz derivative and the parameter $b$ appearing in the solution assume the more simple form
    \begin{equation}
    b= \omega+ c^2|\beta|^{\alpha}.
    \end{equation} 
    Moreover, for $\omega = 0$ the solution obtained in this case reduces to the one given by Garg et al. in \cite{Garg}.
    
    \end{os}
    
    \begin{os}
    We observe that, for $\gamma = 1/2$, $\delta = 3/2$, $f_2(x)= F(x)= \omega= 0$, $\alpha = 2$ and $f_1(x)= \delta(x)$, the solution \eqref{remo} coincides with 
    \begin{align}
    u(x,t)& =\frac{1}{2c\sqrt{\pi t}}\int_{_-\infty}^{+\infty}\bigg\{e^{-\frac{w^2}{4t}-\lambda w}\bigg[\lambda I_0\left(\frac{\lambda}{c}
    \sqrt{c^2w^2-x^2}\right)+\\
	\nonumber  &+\frac{\partial}{\partial w}I_0\left(\frac{\lambda}{c}
	    \sqrt{c^2w^2-x^2}\right)  \bigg]1_{\{|x|<cw\}}\\
	 \nonumber &   +c\left[\delta(x-cw)+\delta(x+cw)\right]\bigg\}dw,
    \end{align}
    (see Theorem 4.2 in \cite{ptrf} for the complete proof).
    This is the probability density of the telegraph process $\mathcal{T}(t)$, $t>0$, with a Brownian time, that is 
    $$W(t)= \mathcal{T}(|B(t|).$$
    In the case $\gamma = 1$, for all $\delta \in [0,1]$, $f_2(x)= F(x)= \omega= 0$, $\alpha = 2$ and $f_1(x)= \delta(x)$, the solution \eqref{remo} coincides with the fundamental solution of the telegraph equation, that is the probability law of the telegraph process.
    Moreover, we observe that, in the case $\gamma= 1$, $\delta\in [0,1]$, $f_1(x)= F(x)=\omega= 0$, $\alpha= 2$ and $f_2(x)= \delta(x)$,
    the Fourier transform of the solution \eqref{remo} becomes
    \begin{equation}
    u^*(\beta,t)= \frac{\left(e^{\xi t}-e^{\eta t}\right)}{\xi-\eta},
    \end{equation}
    which coincides with the characteristic function of the following 
    distribution
    \begin{equation}
    P\{\mathcal{T}(t)\in dx,\bigcup_{k=0}^{\infty} N(t)= 2k+1\}= P\{\mathcal{T}(t)\in dx, N(t)= 2k+1\},
    \end{equation}
    where $N(t)$ is the homogeneous Poisson counting process, that governs
    the changes of direction of the telegraph process. 
    \end{os}

    \begin{os}
    The problem of the existence and uniqueness of the solution for the abstract fractional telegraph process has been considered by Herzallah and Baleanu in \cite{esis}. By simple adaptation of their arguments, the existence and uniqueness of the solution 
    to the Cauchy problem in Theorem 4.1 can be proved. 
    \end{os}

    \end{document}